\documentclass[12pt,twoside]{amsart}
\usepackage{amssymb,amscd,amsxtra,calc}
\usepackage{amsmath}
\usepackage{xypic}

\title{Semi-terminal modifications of demi-normal pairs}
\author{Kento Fujita} 
\date{\today}
\address{Research Institute for Mathematical Sciences, 
Kyoto University, Kyoto 606-8502 Japan}
\email{fujita@kurims.kyoto-u.ac.jp}

\newcommand{\Z}{\mathbb{Z}}
\newcommand{\Q}{\mathbb{Q}}
\newcommand{\R}{\mathbb{R}}
\newcommand{\C}{\mathbb{C}}

\newcommand{\B}{\mathbb{B}}
\newcommand{\A}{\mathbb{A}}
\newcommand{\NC}{\operatorname{N}_1}

\newcommand{\NE}{\operatorname{NE}}

\newcommand{\Supp}{\operatorname{Supp}}

\newcommand{\Spec}{\operatorname{Spec}}

\newcommand{\Proj}{\operatorname{Proj}}
\newcommand{\codim}{\operatorname{codim}}

\newcommand{\cond}{\operatorname{cond}}

\newcommand{\totaldiscrep}{\operatorname{totaldiscrep}}
\newcommand{\discrep}{\operatorname{discrep}}
\newcommand{\Diff}{\operatorname{Diff}}
\newcommand{\Center}{\operatorname{center}}

\newcommand{\sO}{\mathcal{O}}

\newcommand{\sL}{\mathcal{L}}

\newcommand{\Sc}{\operatorname{sc}}
\newcommand{\St}{\operatorname{st}}

\newtheorem{thm}{Theorem}[section]
\newtheorem{lemma}[thm]{Lemma}
\newtheorem{proposition}[thm]{Proposition}

\newtheorem{claim}[thm]{Claim}

\theoremstyle{definition}
\newtheorem{definition}[thm]{Definition}
\newtheorem{remark}[thm]{Remark}

\newtheorem{example}[thm]{Example}
\newtheorem*{ack}{Acknowledgments}

\begin{document}

\maketitle 

\begin{abstract}
For a quasi-projective demi-normal pair $(X, \Delta)$, we prove that there exists a 
semi-canonical modification and a semi-terminal modification of $(X, \Delta)$.
\end{abstract}

\tableofcontents

\section{Introduction}\label{intro_section}

It is a very classical result that for a normal algebraic surface $X$ there exists a 
\emph{minimal resolution} 
$\pi\colon Y\to X$. The morphism $\pi$ is a projective and birational 
morphism such that $Y$ is smooth (i.e., $Y$ has terminal singularities) and there is 
no $(-1)$-curve over $X$ (i.e., $K_Y$ is nef over $X$). This result is generalized in 
\cite{BCHM}; for any normal variety $X$ there exists a 
\emph{terminal modification} $\pi\colon Y\to X$. More precisely, 
the morphism $\pi$ is a projective and birational 
morphism such that $Y$ has terminal singularities and $K_Y$ is nef over $X$. 
On the other hand, in \cite[\S 4]{KSB}, it is known that for any demi-normal 
(see Definition \ref{demi_dfn}) surface there exists a \emph{minimal-semi-resolution} 
$\pi\colon Y\to X$. The morphism $\pi$  is a projective and birational 
morphism such that $Y$ is semi-smooth (see Definition \ref{sing_dfn}), 
$\pi$ is an isomorphism outside a finite set over $X$ and is an isomorphism at any 
generic point of the double curve $D_Y$ of $Y$, and there is 
no $(-1)$-curve on the normalization of $Y$ over $X$ (this implies that 
$K_Y$ is nef over $X$). 

In this article, we consider the reducible version of terminal modifications 
of normal varieties. 
In other words, we consider the higher-dimensional version of minimal-semi-resolutions 
of demi-normal surfaces.
We introduce the notion of \emph{semi-terminal}. 
This notion is a direct generalization of 
semi-smooth surface singularities (see Definition \ref{scst_dfn} \eqref{scst_dfn2} 
and Example \ref{scst_ex} \eqref{scst_ex3}). 
The following is the main result of this article (for the definitions of semi-canonical 
modification and semi-terminal modification, see Definition \ref{model_dfn}).

\begin{thm}[Main Theorem]\label{mainthm}
Let $(X, \Delta)$ be a demi-normal pair. $($We do not assume that $K_X+\Delta$ 
is $\Q$-Cartier.$)$
\begin{enumerate}
\renewcommand{\theenumi}{\arabic{enumi}}
\renewcommand{\labelenumi}{(\theenumi)}
\item\label{mainthm1}
There exists a semi-canonical modification 
$f^{\Sc}\colon X^{\Sc}\to X$ of $(X, \Delta)$ and is unique. 
\item\label{mainthm2}
If $X$ is quasi-projective, then there exists a semi-terminal modification 
$f^{\St}\colon X^{\St}\to X$ of $(X, \Delta)$. Moreover, the morphism $f^{\St}$ 
may be chosen to be projective. 
\end{enumerate}
\end{thm}

\begin{remark}\label{mainthm_rmk}
If $(X, \Delta)$ is a normal pair, then Theorem \ref{mainthm} is obtained by 
\cite{BCHM}. 
If $X$ is a demi-normal surface and $\Delta=0$, then Theorem \ref{mainthm} is 
obtained by \cite[\S 4]{KSB}.
\end{remark}

\begin{remark}\label{mainthm_rmk2}
There are many semi-terminal modifications for a given non-normal 
demi-normal pair $(X, \Delta)$. For example, let $X:=(x_1x_2=0)\subset\A^{3}$ and 
let $\pi\colon\tilde{X}\to X$ be the blowing up at the origin. Then the pairs $(X, 0)$ 
and $(\tilde{X}, 0)$ are semi-terminal and $K_{\tilde{X}}$ is nef over $X$. Hence 
both the identity morphism and $\pi$ are semi-terminal modifications of $(X, 0)$. 
Therefore, for a demi-normal surface $X$, the notion of 
semi-terminal modification of $(X, 0)$ is much weaker than the notion of 
minimal-semi-resolution of $X$. 
\end{remark}

Now we organize the strategy of the proof of Theorem \ref{mainthm}. 
For Theorem \ref{mainthm} \eqref{mainthm1}, the argument is essentially same as 
that of \cite{OX}; taking the normalization, taking the canonical modification and gluing 
along the conductor divisors. For a demi-normal pair $(X, \Delta)$, 
the authors of \cite{OX} remark that 
if $K_X+\Delta$ is not $\Q$-Cartier then the semi-log-canonical modification of 
$(X, \Delta)$ does not exist in general (see \cite[Example 1.40]{SingBook} and 
\cite[Example 3.1]{OX}). However, we remark that 
Theorem \ref{mainthm} \eqref{mainthm1} says that 
for \emph{every} demi-normal pair $(X, \Delta)$ 
(without the assumption $K_X+\Delta$ is $\Q$-Cartier),
there exists the semi-canonical modification of $(X, \Delta)$. 

The strategy of the proof of Theorem \ref{mainthm} \eqref{mainthm2} is the following. 
First, for a given demi-normal pair $(X, \Delta)$, we take a semi-log-resolution 
$Y\to X$. Then we run a \emph{MMP with scaling over $X$ for reducible varieties}. 
It is known that the Contraction theorem for semi-log-canonical pairs was 
established in \cite{fujino}. However, as in \cite[Example 5.4]{fujino},
no possible minimal model program for reducible varieties has been known in general 
(at least in absolute setting). 
Our strategy is taking the \emph{semi-canonical modification} instead of taking 
the flip for the contraction morphism. 
Then the program can be run in this case. 
Finally, we decompose each step of the 
program and show that the program terminates (cf.\ \cite[\S 4.2]{BOOK}). 
This is the strategy of the proof of Theorem \ref{mainthm} \eqref{mainthm2}.

\begin{ack}
The author would like to thank Professors Shigefumi Mori, Shigeru Mukai, 
Noboru Nakayama, Osamu Fujino, Masayuki Kawakita and Stefan Helmke 
for comments during the seminars in RIMS. 
Especially, Professor Masayuki Kawakita helps the author 
for improving the notion of ``semi-terminal" and Professor Osamu Fujino helps the 
author for improving the proof of Lemma \ref{base_lem}. 
The author learned the theory of demi-normal varieties from 
Professor J\'anos Koll\'ar during the author visited Princeton University. 
The author is partially supported by a JSPS Fellowship for Young Scientists. 
\end{ack}

Throughout this paper, we will work over the complex number field $\C$. 
In this paper, a \emph{variety} means a reduced algebraic (separated 
and of finite type) scheme over $\C$. 
For a morphism $f\colon Y\to X$ between equidimensional varieties, 
the morphism $f$ is said to be \emph{an isomorphism in codimension $1$ over $X$} 
if there exists an open subscheme $U\subset X$ such that 
$\codim_X(X\setminus U)\geq 2$ and $f\colon f^{-1}(U)\to U$ is an isomorphism. 
For a variety $X$, the normalization of $X$ is 
denoted by $\nu_X\colon\bar{X}\to X$.
For the theory of minimal model program (\emph{MMP}, for short), 
we refer the readers to \cite{KoMo} and \cite{SingBook}. 

\section{Preliminaries}\label{prelim_section}

In this section, we collect some basic definitions and results. 

\begin{definition}\label{sing_dfn}
\begin{enumerate}
\renewcommand{\theenumi}{\arabic{enumi}}
\renewcommand{\labelenumi}{(\theenumi)}
\item\label{sing_dfn1}
Let $X$ be a variety and let $x\in X$ be a closed point. 
We say that $x\in X$ is a \emph{double normal crossing} (\emph{dnc}, for short) point 
if $\hat{\sO}_{X, x}\simeq\C[[x_0,\dots,x_n]]/(x_0x_1)$; a \emph{pinch} point 
if $\hat{\sO}_{X, x}\simeq\C[[x_0,\dots,x_n]]/(x_0^2-x_1^2x_2)$, respectively. 
\item\label{sing_dfn2}
A variety $X$ is said to be a \emph{double normal crossing} variety (\emph{dnc} variety, 
for short) if any closed point $x\in X$ is either smooth or dnc point; 
a \emph{semi-smooth} variety if any closed point $x\in X$ is one of smooth, 
dnc or pinch point, respectively.
\end{enumerate}
\end{definition}

\begin{definition}[{\cite[\S 5.1]{SingBook}}]\label{demi_dfn}
\begin{enumerate}
\renewcommand{\theenumi}{\arabic{enumi}}
\renewcommand{\labelenumi}{(\theenumi)}
\item\label{demi_dfn1}
Let $X$ be an equidimensional variety. We say that $X$ is a \emph{demi-normal} 
variety if $X$ satisfies Serre's $S_2$ condition and is dnc outside codimension $2$. 
\item\label{demi_dfn2}
For an equidimensional variety $X$, if $X$ is dnc outside codimension $2$, then 
there exists a unique finite and birational morphism $d\colon X^d\to X$ such that 
$X^d$ is a demi-normal variety and $d$ is an isomorphism 
in codimension $1$ over $X$. We call $d$ the \emph{demi-normalization} 
of $X$.
\item\label{demi_dfn3}
Let $X$ be a demi-normal variety and let $\nu_X\colon \bar{X}\to X$
be its normalization. 
Then the \emph{conductor ideal} of $X$ is defined by 
$\cond_X:=\mathcal Hom_{\sO_X}((\nu_X)_*\sO_{\bar{X}}, \sO_X)\subset\sO_X$. 
This can be seen as an ideal sheaf $\cond_{\bar{X}}$ on $\bar{X}$. 
Let $D_X:=\Spec_X(\sO_X/\cond_X)$ and 
$D_{\bar{X}}:=\Spec_{\bar{X}}(\sO_{\bar{X}}/\cond_{\bar{X}})$. 
We call $D_X$ (resp.\ $D_{\bar{X}}$) as the \emph{conductor divisor} of $X$ 
(resp.\ of $\bar{X}/X$). It is known that both $D_X$ and $D_{\bar{X}}$ are reduced 
and of pure codimension $1$. Moreover, for the normalization 
$\nu_{D_{\bar{X}}}\colon\bar{D}_{\bar{X}}\to D_{\bar{X}}$, we get the Galois involution 
$\iota_X\colon\bar{D}_{\bar{X}}\to\bar{D}_{\bar{X}}$ defined by $\nu_X$. 
\end{enumerate}
\end{definition}

\begin{definition}\label{scst_dfn}
\begin{enumerate}
\renewcommand{\theenumi}{\arabic{enumi}}
\renewcommand{\labelenumi}{(\theenumi)}
\item\label{scst_dfn1}
The pair $(X, \Delta)$ is said to be a \emph{demi-normal pair} if 
$X$ is a demi-normal variety and $\Delta$ is a formal $\Q$-linear combination
$\Delta=\sum_{i=1}^ka_i\Delta_i$ of irreducible and reduced closed subvarieties 
$\Delta_i$ of codimension $1$ such that $\Delta_i\not\subset\Supp D_X$ and 
$a_i\in[0, 1]\cap\Q$ for any $1\leq i\leq k$. Furthermore, if $X$ is a normal variety, 
then the pair $(X, \Delta)$ is said to be a \emph{normal pair}.
\item\label{scst_dfn2}
Let $(X, \Delta)$ be a demi-normal pair and let $\nu_X\colon\bar{X}\to X$ be the 
normalization of $X$. Set $\Delta_{\bar{X}}:=(\nu_X)^{-1}_*\Delta$. 
\begin{enumerate}
\renewcommand{\theenumii}{\roman{enumii}}
\renewcommand{\labelenumii}{(\theenumii)}
\item\label{scst_dfn21}
\cite[Definition 4.17]{KSB}
The pair $(X, \Delta)$ is \emph{semi-canonical} if $K_X+\Delta$ is $\Q$-Cartier 
and the pair $(\bar{X}, \Delta_{\bar{X}}+D_{\bar{X}})$ has canonical singularities.
\item\label{scst_dfn22}
The pair $(X, \Delta)$ is \emph{semi-terminal} if $(X, \Delta)$ is semi-canonical 
and for any exceptional divisor $E$ over $\bar{X}$ the inequality 
$a(E, \bar{X}, \Delta_{\bar{X}}+D_{\bar{X}})>0$ holds unless 
$\Center_{\bar{X}}E\subset\lfloor\Delta_{\bar{X}}+D_{\bar{X}}\rfloor$ and 
$\codim_{\bar{X}}(\Center_{\bar{X}}E)=2$. 
\end{enumerate}
\end{enumerate}
\end{definition}

\begin{remark}\label{scstdfn_rmk}
Let $(Y, \Delta+S)$ be a normal pair with $S=\lfloor S\rfloor$. 
\begin{enumerate}
\renewcommand{\theenumi}{\arabic{enumi}}
\renewcommand{\labelenumi}{(\theenumi)}
\item\label{scstdfn_rmk1}
If the pair $(Y, \Delta+S)$ has canonical singularities, then 
$\Diff_{S}\Delta=0$ and the pair $(S, 0)$ has canonical singularities. 
In particular, $S$ is normal. 
\item\label{scstdfn_rmk2}
If the pair $(Y, \Delta+S)$ is semi-terminal, then 
the pair $(S, 0)$ has terminal singularities. 
\end{enumerate}
In particular, for a demi-normal pair $(X, \Delta)$, the following holds. 
(1) If $(X, \Delta)$ is semi-canonical then the pair  
$(\lfloor\Delta_{\bar{X}}+D_{\bar{X}}\rfloor, 0)$ has canonical singularities. 
(2) If $(X, \Delta)$ is semi-terminal then the pair  
$(\lfloor\Delta_{\bar{X}}+D_{\bar{X}}\rfloor, 0)$ has terminal singularities. 
\end{remark}

\begin{proof}
\eqref{scstdfn_rmk1}
Since the pair $(Y, \Delta+S)$ is plt, $S$ is normal by \cite[Proposition 5.51]{KoMo}. 
Moreover, by adjunction, 
\[
\totaldiscrep(S, \Diff_S\Delta)\geq\discrep(\text{center}\subset S, Y, \Delta+S)\geq 0
\]
holds (see \cite[Lemma 4.8]{SingBook}). Thus $\Diff_S\Delta=0$ and the pair 
$(S, 0)$ has canonical singularities.

\eqref{scstdfn_rmk2}
Assume that the pair $(S, 0)$ does not have terminal singularities. 
Then there exists an exceptional divisor $E_S$ over $S$ such that 
$a(E_S, S, 0)=0$ by \eqref{scstdfn_rmk1}. 
We note that $\codim_Y Z\geq 3$, where $Z:=\Center_S E_S$. 
By adjunction,
\[
\totaldiscrep(\Center\subset Z, S, \Diff_S\Delta)
\geq\discrep(\Center\subset Z, Y, \Delta+S)
\]
holds. Moreover, by the fact that the pair $(Y, \Delta+S)$ is semi-terminal, 
the right-hand of the above inequality is positive. 
However, the left-hand of the above inequality is less than or equal to $a(E_S, S, 0)=0$. 
This leads to a contradiction. Thus the pair $(S, 0)$ has terminal singularities. 
\end{proof}

\begin{example}\label{scst_ex}
\begin{enumerate}
\renewcommand{\theenumi}{\arabic{enumi}}
\renewcommand{\labelenumi}{(\theenumi)}
\item\label{scst_ex1}
\cite[Corollary 2.31]{KoMo}
If $(X, \Delta)$ is a normal pair such that $X$ is a smooth variety and 
$\Supp\Delta\subset X$ is a smooth divisor, then the pair $(X, \Delta)$
is semi-terminal. 
\item\label{scst_ex2}
If $(X, \Delta)$ is a demi-normal pair such that $X$ is a semi-smooth variety, 
$\Supp\Delta$ is contained in the smooth locus of $X$ and $\Supp\Delta$ is 
a smooth divisor, then the pair $(X, \Delta)$ is semi-terminal by \eqref{scst_ex1}.
\item\label{scst_ex3}\cite[Proposition 4.12]{KSB}
Let $X$ be a demi-normal surface and $x\in X$ be a closed point. 
The pair $(X, 0)$ is semi-canonical around $x$ if and only if $x\in X$ is either 
smooth, du Val, dnc or pinch point.
The pair $(X, 0)$ is semi-terminal around $x$ if and only if $X$ is semi-smooth 
around $x$. Thus the notion of semi-terminal singularities is a direct generalization of 
the notion of semi-smooth surface singularities. 
\end{enumerate}
\end{example}

\begin{definition}\label{model_dfn}
Let $(X, \Delta)$ be a demi-normal pair and let $f\colon Y\to X$ be a proper 
birational morphism such that $Y$ is a demi-normal variety, $f$ is an isomorphism 
in codimension $1$ over $X$ and 
$f$ is an isomorphism around any generic point of $D_Y$. 
Set $\Delta_Y:=f^{-1}_*\Delta$.
\begin{enumerate}
\renewcommand{\theenumi}{\arabic{enumi}}
\renewcommand{\labelenumi}{(\theenumi)}
\item\label{model_dfn1}
The morphism $f$ is said to be a \emph{semi-canonical modification} of $(X, \Delta)$ 
if $(Y, \Delta_Y)$ is semi-canonical and $K_Y+\Delta_Y$ is ample over $X$. 
Furthermore, if $X$ (and also $Y$) is a normal variety, then such $f$ is called 
a \emph{canonical modification} of $(X, \Delta)$.  
\item\label{model_dfn2}
The morphism $f$ is said to be a \emph{semi-terminal modification} of $(X, \Delta)$ 
if $(Y, \Delta_Y)$ is semi-terminal and $K_Y+\Delta_Y$ is nef over $X$. 
\end{enumerate}
\end{definition}

\section{Semi-canonical modifications}\label{sc_section}

The following lemma and proposition are proven essentially 
same as \cite[Lemma 2.1 and Proposition 2.2]{OX}. 

\begin{lemma}[{cf.\ \cite[Lemma 2.1]{OX}}]\label{OX22}
Let $(X, \Delta)$ be a normal pair and let $\tilde{f}\colon\tilde{Y}\to X$ 
be a projective log resolution of $(X, \Delta)$ such that 
$\Supp\Delta_{\tilde{Y}}\subset\tilde{Y}$ is a smooth divisor, 
where $\Delta_{\tilde{Y}}:=\tilde{f}^{-1}_*\Delta$. 
If the pair $(\tilde{Y}, \Delta_{\tilde{Y}})$ 
has a canonical model $(Y, \Delta_Y)$ over $X$ 
$($for the definition of canonical models of pairs, see \cite[\S 3.8]{KoMo}$)$, 
then the morphism $Y\to X$ is a canonical modification of $(X, \Delta)$. 
\end{lemma}

\begin{proof}
It is enough to show that the pair $(Y, \Delta_Y)$ has canonical singularities. 
Take any exceptional divisor $E$ over $Y$. 
Then $E$ is either an exceptional divisor over $\tilde{Y}$ or a divisor on $\tilde{Y}$ 
with $E\not\subset\Supp\Delta_{\tilde{Y}}$. 
Thus $a(E, \tilde{Y}, \Delta_{\tilde{Y}})\geq 0$ holds since the pair 
$(\tilde{Y}, \Delta_{\tilde{Y}})$ has canonical singularities. 
By \cite[Proposition 3.51]{KoMo}, 
$a(E, Y, \Delta_Y)\geq a(E, \tilde{Y}, \Delta_{\tilde{Y}})\geq 0$ holds. 
Hence the pair $(Y, \Delta_Y)$ has canonical singularities. 
\end{proof}

\begin{proposition}[{cf.\ \cite[Proposition 2.2]{OX}}]\label{OX23}
Let $(X, \Delta)$ be a normal pair. Then a canonical modification of $(X, \Delta)$ is 
unique, if exists. 
\end{proposition}

\begin{proof}
Let $f\colon Y\to X$ be a canonical modification of the pair $(X, \Delta)$. 
Let $g\colon\tilde{Y}\to Y$ be an arbitrary log resolution of the pair $(Y, \Delta_Y)$ 
such that $\Supp\Delta_{\tilde{Y}}\subset Y$ is a smooth divisor, where 
$\Delta_Y:=f^{-1}_*\Delta$ and $\Delta_{\tilde{Y}}:=g^{-1}_*\Delta_{Y}$. 
Set $\tilde{f}:=f\circ g\colon\tilde{Y}\to X$. 
Since the pair $(Y, \Delta_Y)$ has canonical singularities, 
we can write $K_{\tilde{Y}}+\Delta_{\tilde{Y}}=g^*(K_Y+\Delta_Y)+F$ such that 
$F$ is an effective exceptional divisor over $Y$. 
Therefore $Y$ is isomorphic to 
\[
\Proj_X\bigoplus_{m\geq 0}f_*
\sO_Y(\lfloor m(K_Y+\Delta_Y)\rfloor)\simeq
\Proj_X\bigoplus_{m\geq 0}\tilde{f}_*
\sO_{\tilde{Y}}(\lfloor m(K_{\tilde{Y}}+\Delta_{\tilde{Y}})\rfloor).
\]
Therefore a canonical modification of $(X, \Delta)$ is unique.
\end{proof}

We recall the results in \cite{BCHM}.

\begin{thm}[{\cite{BCHM}}]\label{BCHM_thm}
Let $(X, \Delta)$ be a quasi-projective normal $\Q$-factorial dlt pair and 
let $f\colon X\to U$ be a projective morphism 
between normal quasi-projective varieties. We assume that $\Delta$ and $K_X+\Delta$ 
are big over $U$ and $\B_+(\Delta/U)$ does not contain any lc center of $(X, \Delta)$, 
where $\B_+(\Delta/U)$ is the augmented base locus of $\Delta$ over $U$ 
$($see \cite[Definition 3.5.1]{BCHM}$)$. Then any $(K_X+\Delta)$-MMP with 
ample scaling over $U$ induces a good minimal model $(X^m, \Delta^m)$ over $U$, 
that is, $K_{X^m}+\Delta^m$ is semiample over $U$. 
\end{thm}

\begin{proof}
By \cite[Lemma 3.7.5]{BCHM}, there exists an effective $\Q$-divisor $\Delta'$ 
such that $K_X+\Delta\sim_{\Q, U}K_X+\Delta'$ and the pair $(X, \Delta')$ is klt. 
Thus we may run $(K_X+\Delta)$-MMP with ample scaling over $U$ and terminates 
by \cite[Corollary 1.4.2]{BCHM}. Since $K_X+\Delta$ is big over $U$, this MMP induces 
a good minimal model over $U$ by the Basepoint-free theorem 
\cite[Theorem 3.24]{KoMo}.
\end{proof}

The following lemma is well-known. 

\begin{lemma}\label{base_lem}
Let $(X, \Delta)$ be a quasi-projective normal pair and let 
$f\colon Y\to X$ be a projective log resolution of $(X, \Delta)$. 
Set $\Delta_Y:=f^{-1}_*\Delta$. 
Then we have the following:
\begin{enumerate}
\renewcommand{\theenumi}{\arabic{enumi}}
\renewcommand{\labelenumi}{(\theenumi)}
\item\label{base_lem1}
Any $\Q$-divisor on $Y$ is big over $X$.
\item\label{base_lem2}
The augmented base locus $\B_+(\Delta_Y/X)$ does not contain 
any irreducible component of $\Supp\Delta_Y$.
\end{enumerate}
\end{lemma}

\begin{proof}
\eqref{base_lem1} is obvious. We prove \eqref{base_lem2}. 
Take a divisor $A$ on $Y$ which is ample over $X$, take a sufficiently small 
rational number $0<\epsilon\ll 1$, and take $m\in\Z_{>0}$ such that 
$m(\Delta_Y-\epsilon A)$ is Cartier. Then the sheaf 
$f_*\sO_Y(m(\Delta_Y-\epsilon A))$ is of rank one. 
Take a general global section of the sheaf 
$f_*\sO_Y(m(\Delta_Y-\epsilon A))\otimes\sO_X(lH)$, where $H$ is ample on $X$ 
and $l\gg 0$. 
Then the pullback of the global section on $Y$ does not contain 
any irreducible component $S$ of $\Supp\Delta_Y$ since $f$ is an 
isomorphism at the generic point of $S$. 
Thus $S\not\subset\B_+(\Delta_Y/X)$ holds.
\end{proof}

As a corollary, we get the following theorem of \cite{BCHM}. 
We give a proof for the reader's convenience. 
We remark that this theorem is a direct consequence of 
\cite[Corollary 1.4.2 and Lemma 3.7.5]{BCHM}. See also \cite[Theorem 1.31]{SingBook}.

\begin{thm}[{\cite{BCHM}}]\label{canmodel_thm}
For any normal pair $(X, \Delta)$, there exists a canonical modification $f\colon Y\to X$ 
of $(X, \Delta)$ and is unique. 
\end{thm}

\begin{proof}
By Proposition \ref{OX23}, we can assume that $X$ is quasi-projective. 
Let $\tilde{f}\colon\tilde{Y}\to X$ be a projective log resolution of $(X, \Delta)$ 
such that $\Supp\Delta_{\tilde{Y}}\subset\tilde{Y}$ is a smooth divisor, 
where $\Delta_{\tilde{Y}}:=\tilde{f}^{-1}_*\Delta$. 
Then the pair $(\tilde{Y}, \Delta_{\tilde{Y}})$ is $\Q$-factorial 
and has canonical singularities. 
Hence the set of lc centers of $(\tilde{Y}, \Delta_{\tilde{Y}})$ is equal to 
the set of irreducible components of $\lfloor\Delta_{\tilde{Y}}\rfloor$. 
By Lemma \ref{base_lem}, any irreducible component of $\lfloor\Delta_{\tilde{Y}}\rfloor$ 
is not contained in $\B_+(\Delta_{\tilde{Y}}/X)$. 
Thus we can run $(K_{\tilde{Y}}+\Delta_{\tilde{Y}})$-MMP with ample scaling over $X$ 
and induces a good minimal model over $X$
by Theorem \ref{BCHM_thm}. 
Hence there exists the canonical model $Y$ of $(\tilde{Y}, \Delta_{\tilde{Y}})$ 
over $X$. By Lemma \ref{OX22}, the morphism $Y\to X$ is the canonical modification 
of $(X, \Delta)$. 
\end{proof}

\begin{proposition}[{cf.\ \cite[Corollary 2.1]{OX}}]\label{OX29}
Let $(X, \Delta+S)$ be a normal pair with $S=\lfloor S\rfloor$ and 
let $f\colon Y\to X$ be the canonical modification of $(X, \Delta+S)$. 
Set $S_Y:=f^{-1}_*S$ and let $\nu_S\colon \bar{S}\to S$ be the normalization. 
Then the morphism $f$ induces the birational morphism 
$f_{\bar{S}}\colon S_Y\to\bar{S}$ and the morphism $f_{\bar{S}}$ is 
the canonical modification of $(\bar{S}, 0)$. 
\end{proposition}

\begin{proof}
By Remark \ref{scstdfn_rmk}, the pair $(S_Y, 0)$ has canonical singularities 
(in particular, $S_Y$ is normal) and $K_{S_Y}=(K_Y+\Delta_Y+S_Y)|_{S_Y}$ is ample 
over $\bar{S}$. 
Thus the morphism $f_{\bar{S}}$ is the canonical modification of $(\bar{S}, 0)$.
\end{proof}

\begin{lemma}[{cf.\ \cite[Lemma 3.1]{OX}}]\label{OX33}
Let $(X, \Delta)$ be a demi-normal pair. 
\begin{enumerate}
\renewcommand{\theenumi}{\arabic{enumi}}
\renewcommand{\labelenumi}{(\theenumi)}
\item\label{OX331}
A semi-canonical modification of $(X, \Delta)$ is unique, if exists. 
\item\label{OX332}
Let $f\colon Y\to X$ be the semi-canonical modification of $(X, \Delta)$, 
let $\nu_Y\colon\bar{Y}\to Y$ and $\nu_X\colon\bar{X}\to X$ be the normalizations 
and let $\bar{f}\colon\bar{Y}\to\bar{X}$ be the morphism obtained by $f$. 
Then the morphism $\bar{f}$ is the canonical modification of 
$(\bar{X}, \Delta_{\bar{X}}+D_{\bar{X}})$, where 
$\Delta_{\bar{X}}:=(\nu_X)^{-1}_*\Delta$ and $D_{\bar{X}}$ is the conductor 
divisor of $\bar{X}/X$.
\end{enumerate}
\end{lemma}

\begin{proof}
Let $f\colon Y\to X$ be a semi-canonical modification of $(X, \Delta)$. 
Then $K_{\bar{Y}}+\Delta_{\bar{Y}}+D_{\bar{Y}}=\nu_Y^*(K_Y+\Delta_Y)$ is ample 
over $X$ and the pair $(\bar{Y}, \Delta_{\bar{Y}}+D_{\bar{Y}})$ has 
canonical singularities. 
Hence the morphism $\bar{f}$ is the canonical modification 
of $(\bar{X}, \Delta_{\bar{X}}+D_{\bar{X}})$. Thus we get \eqref{OX332} and 
$\bar{Y}$ is unique. On the other hand, the Galois involution 
$\iota_X\colon\bar{D}_{\bar{X}}\to\bar{D}_{\bar{X}}$ is extended to 
$\iota\colon D_{\bar{Y}}\to D_{\bar{Y}}$ uniquely, where $\bar{D}_{\bar{X}}$ is
the normalization of $D_{\bar{X}}$. 
Thus the quotient $\bar{Y}\to Y$ by $\iota$ is unique 
by \cite[Proposition 5.3]{SingBook}. 
\end{proof}

\begin{thm}[{=Theorem \ref{mainthm} \eqref{mainthm1}}]\label{mainthm_1}
For any demi-normal pair $(X, \Delta)$, the canonical modification 
of $(X, \Delta)$ exists and is unique $($up to isomorphism over $X$$)$. 
\end{thm}

\begin{proof}
Let $\nu_X\colon\bar{X}\to X$ be the normalization and let 
$\Delta_{\bar{X}}:=(\nu_X)^{-1}_*\Delta$. 
By Theorem \ref{canmodel_thm}, there exists the canonical modification 
$\bar{f}\colon\bar{Y}\to\bar{X}$ of $(\bar{X}, \Delta_{\bar{X}}+D_{\bar{X}})$. 
Set $\Delta_{\bar{Y}}:=\bar{f}^{-1}_*\Delta_{\bar{X}}$ and 
$D_{\bar{Y}}:=\bar{f}^{-1}_*D_{\bar{X}}$. 
By Proposition \ref{OX29}, the morphism 
$\bar{f}_{\bar{D}_{\bar{X}}}\colon D_{\bar{Y}}\to\bar{D}_{\bar{X}}$ is the canonical 
modification of $(\bar{D}_{\bar{X}}, 0)$, where $\bar{D}_{\bar{X}}$ is
the normalization of $D_{\bar{X}}$. 
Hence the involution $\iota_X\colon\bar{D}_{\bar{X}}\to\bar{D}_{\bar{X}}$ can be 
extended to the involution $\iota\colon D_{\bar{Y}}\to D_{\bar{Y}}$. 
Since $K_{\bar{Y}}+\Delta_{\bar{Y}}+D_{\bar{Y}}$ is ample over $X$, 
there exists a semi-canonical pair $(Y, \Delta_Y)$ over $X$ such that 
the normalization of $Y$ and $D_{\bar{Y}}$ is exactly same as the 
conductor divisor of $\bar{Y}/Y$ by 
\cite[Corollary 5.37, Corollary 5.33 and Theorem 5.38]{SingBook}. 
The morphism $Y\to X$ is exactly the semi-canonical modification of $(X, \Delta)$.
\end{proof}

\section{Semi-terminal modifications}\label{st_section}

In this section, we prove Theorem \ref{mainthm} \eqref{mainthm2}. 
Let $(X, \Delta)$ be an arbitrary quasi-projective demi-normal pair. 
We show that there exists a projective semi-terminal modification of $(X, \Delta)$.

\subsection{Semi-log-resolution}\label{st_step1}

By \cite[Theorem 10.54]{SingBook}, there exists a projective and birational morphism 
$f\colon Y\to X$ such that the following properties hold:
\begin{enumerate}
\renewcommand{\theenumi}{\roman{enumi}}
\renewcommand{\labelenumi}{(\theenumi)}
\item\label{slr1}
$Y$ is semi-smooth. 
\item\label{slr2}
$f$ is an isomorphism in codimension $1$ over $X$ and $f$ is an isomorphism 
at any generic point of $D_Y$. 
\item\label{slr3}
$\Supp\Delta_Y$ is contained in the smooth locus of $Y$ and 
$\Supp\Delta_Y$ is a smooth divisor, where $\Delta_Y:=f^{-1}_*\Delta$. 
\item\label{slr4}
Let $\bar{f}\colon\bar{Y}\to\bar{X}$ be the morphism obtained by the normalizations 
$\nu_X\colon\bar{X}\to X$ and $\nu_Y\colon\bar{Y}\to Y$. 
Then the morphism $\bar{f}$ is a projective log resolution of the pair 
$(\bar{X}, \Delta_{\bar{X}}+D_{\bar{X}})$, where 
$\Delta_{\bar{X}}$ is the strict transform of $\Delta$.
\end{enumerate}

We fix a Cartier divisor $H$ on $Y$ which is ample over $X$ 
such that $K_Y+\Delta_Y+H$ is nef over $X$.

\subsection{Running a reducible MMP with scaling}\label{st_step2}

In this section, we will construct inductively the following (for $i\geq 0$):

\begin{enumerate}
\renewcommand{\theenumi}{\arabic{enumi}}
\renewcommand{\labelenumi}{(\theenumi)}
\item\label{smmp1}
Projective and birational morphisms 
\[
Y_i\xrightarrow{\pi_i}W_i\xleftarrow{d_i}W_i^d\xleftarrow{\pi_i^+}Y_{i+1}
\]
over $X$ such that 
all of them are isomorphisms in codimension $1$ over its images and are isomorphisms 
at all generic points of the conductor divisors. 
\item\label{smmp2}
A rational number $\lambda_i$ such that 
$0<\lambda_i\leq 1$ and $\lambda_0\geq\lambda_1\geq\cdots$.
\item\label{smmp3}
A $\Q$-Cartier $\Q$-divisor $H_i$ on $Y_i$, a 
positive integer $l_i$ and 
an invertible sheaf $\sL_i$ on $W_i$ which is nef over $X$. 
\end{enumerate}
The properties are the following: 
\begin{enumerate}
\renewcommand{\theenumi}{\roman{enumi}}
\renewcommand{\labelenumi}{(\theenumi)}
\item\label{smmp4}
$Y_0=Y$ and $H_0=H$ holds. 
The pair $(Y_i, \Delta_i)$ is semi-terminal, 
where $\Delta_i$ is the strict transform of $\Delta$. 
\item\label{smmp5}
The following holds. 
\[
\lambda_i=\inf\{\lambda\in\R_{\geq 0}\, |\, K_{Y_i}+\Delta_i+\lambda H_i 
\text{ is nef over }X\}.
\]
The morphism $\pi_i$ is the contraction morphism associated to a 
$(K_{Y_i}+\Delta_i)$-negative extremal ray $R_i\subset\overline{\NE}(Y_i/X)$ such that 
$(K_{Y_i}+\Delta_i+\lambda_iH_i\cdot R_i)=0$. 
\item\label{smmp6}
The morphism $d_i$ is the demi-normalization of $W_i$. 
The morphism $\pi_i^+$ is the semi-canonical modification of $(W_i^d, \Delta_{W_i^d})$, 
where $\Delta_{W_i^d}$ is the strict transform of $\Delta$. 
\item\label{smmp7}
The following holds:
\begin{eqnarray*}
\pi_i^*\sL_i & \simeq & \sO_{Y_i}(l_i(K_{Y_i}+\Delta_i+\lambda_iH_i)), \\
(d_i\circ\pi_i^+)^*\sL_i & \simeq & \sO_{Y_{i+1}}(l_i(K_{Y_{i+1}}+\Delta_{i+1}+
\lambda_iH_{i+1})).
\end{eqnarray*}
\end{enumerate}

\noindent\textbf{Construction.}
Set $Y_0:=Y$ and $H_0:=H$. 

Assume that we have constructed $Y_i$ and $H_i$ 
(and also $\lambda_{i-1}$, $l_{i-1}$ and $\sL_{i-1}$, if $i\geq 1$). 
If $K_{Y_i}+\Delta_i$ is nef over $X$, then we stop the program and go to 
Section \ref{st_step5}. 

We consider the case that $K_{Y_i}+\Delta_i$ is not nef over $X$. Set 
$\lambda_i$ as in \eqref{smmp5}.
If $i=0$, then $K_{Y_i}+\Delta_i+H_i$ is nef over $X$ by definition. 
If $i\geq 1$, then $(d_{i-1}\circ\pi_{i-1}^+)^*\sL_{i-1}$ is nef over $X$ by induction. 
Hence $K_{Y_i}+\Delta_i+\lambda_{i-1}H_i$ is nef over $X$. 
Thus $0<\lambda_i\leq 1$, and $\lambda_i\leq\lambda_{i-1}$ if $i\geq 1$. 
We can find a $(K_{Y_i}+\Delta_i)$-negative extremal ray 
$R_i\subset\overline{\NE}(Y_i/X)$ with 
$(K_{Y_i}+\Delta_i+\lambda_{i+1}H_i\cdot R_i)=0$, 
and we can get the contraction morphism over $X$ with respect to $R_i$ by 
\cite[Theorem 1.19]{fujino}. 
In particular, $\lambda_i$ is a rational number. 
Let $\pi_i\colon Y_i\to W_i$ be the corresponding 
contraction morphism over $X$. The morphism $\pi_i$ is a projective and birational 
morphism. 
Since $(K_{Y_i}+\Delta_i+\lambda_iH_i\cdot R_i)=0$, 
we can find a positive integer $l_i$ and an invertible sheaf 
$\sL_i$ on $W_i$ such that 
$\pi_i^*\sL_i\simeq\sO_{Y_i}(l_i(K_{Y_i}+\Delta_i+\lambda_iH_i))$ by the Contraction 
theorem \cite[Theorem 1.19]{fujino}. 
Since $\pi_i^*\sL_i$ is nef over $X$, $\sL_i$ is also nef over $X$. 
We can take the demi-normalization $d_i\colon W_i^d\to W_i$ 
since $W_i$ is dnc outside codimension $2$. 
Let $\pi_i^+\colon Y_{i+1}\to W_i^d$ be the semi-canonical modification of 
$(W_i^d, \Delta_{W_i^d})$, where $\Delta_{W_i^d}$ is the strict transform of $\Delta$. 
(We note that $\pi_i^+$ exists and is unique 
by Theorem \ref{mainthm} \eqref{mainthm1}.)
We take a $\Q$-divisor $H_{i+1}$ on $Y_{i+1}$ such that the following holds:
\[
(d_i\circ\pi_i^+)^*\sL_i\simeq\sO_{Y_{i+1}}(l_i(K_{Y_{i+1}}+\Delta_{i+1}+
\lambda_iH_{i+1})).
\]
The $\Q$-divisor $H_{i+1}$ is $\Q$-Cartier since $K_{Y_{i+1}}+\Delta_{i+1}$ is 
$\Q$-Cartier.

\begin{claim}\label{claim1}
The pair $(Y_{i+1}, \Delta_{i+1})$ is semi-terminal. 
\end{claim}

\begin{proof}[{Proof of Claim \ref{claim1}}]
Let $\bar{Y}_i$, $\bar{W_i}$ be the normalization of $Y_i$, $W_i$, respectively. 
We note that $\bar{W}_i$ is equal to the normalization of $W_i^d$ by 
Zariski's Main Theorem. 
Let 
\[
\bar{Y}_i\xrightarrow{\bar{\pi}_i}\bar{W}_i\xleftarrow{\bar{\pi}_i^+}\bar{Y}_{i+1}
\]
be the morphisms obtained by $\pi_i$ and $\pi_i^+$. 
Since the pair $(Y_{i+1}, \Delta_{i+1})$ is semi-canonical, 
it is enough to show that the pair $(\bar{Y}_{i+1}, B_{i+1})$ 
is semi-terminal, where $B_i$ is the sum of $D_{\bar{Y}_i}$ and the strict 
transform of $\Delta_i$. 
We know that $-(K_{\bar{Y}_i}+B_i)$ is ample over $\bar{W}_i$ 
and $K_{\bar{Y}_{i+1}}+B_{i+1}$ is ample over $\bar{W}_i$. 
Take any exceptional divisor $E$ over $\bar{Y}_{i+1}$ such that 
$a(E, \bar{Y}_{i+1}, B_{i+1})=0$ holds. 
It is enough to show that 
$\Center_{\bar{Y}_{i+1}}E\subset\lfloor B_{i+1}\rfloor$ 
and $\codim_{\bar{Y}_{i+1}}(\Center_{\bar{Y}_{i+1}}E)=2$. 
Assume that either $\bar{\pi}_i$ or $\bar{\pi}_i^+$ is not an isomorphism over 
the generic point of $\Center_{\bar{W}_i}E$. Then we have 
\[
a(E, \bar{Y}_{i}, B_i)<a(E, \bar{Y}_{i+1}, B_{i+1})
\]
by the negativity lemma \cite[Lemma 3.38]{KoMo}. 
Since the pair $(\bar{Y}_{i}, B_i)$ has canonical 
singularities, this leads to a contradiction. 
(We remark that $\bar{\pi}_i$, $\bar{\pi}_i^+$ are isomorphisms over 
the images of the generic point of all components of 
$\Supp B_i$, 
$\Supp B_{i+1}$, respectively. 
Thus $E\not\subset\Supp B_i$.)
Therefore both $\bar{\pi}_i$ are $\bar{\pi}_i^+$ are isomorphisms at 
the generic point of $\Center_{\bar{W}_i}E$. In particular, 
$a(E, \bar{Y}_{i}, B_i)=0$ holds. 
Since the pair $(\bar{Y}_{i}, B_i)$ is semi-terminal, 
we have $\Center_{\bar{Y}_{i}}E\subset\lfloor B_i\rfloor$ and 
$\codim_{\bar{Y}_i}(\Center_{\bar{Y}_i}E)=2$. 
Thus we have 
$\codim_{\bar{Y}_{i+1}}(\Center_{\bar{Y}_{i+1}}E)=2$ and 
$\Center_{\bar{Y}_{i+1}}E\subset\lfloor B_{i+1}\rfloor$. 
\end{proof}

Therefore, we can construct the objects in Section \ref{st_step2} 
\eqref{smmp1}, \eqref{smmp2} and \eqref{smmp3} inductively.

\subsection{Decomposing the MMP}\label{st_step3}

Let  $\phi_0\colon Z_{0,0}\to\bar{Y}_0$ be the identity morphism and 
let $H_{0,0}:=(\nu_{Y_0}\circ\phi_0)^*H_0$, where $\nu_{Y_0}\colon\bar{Y}_0\to Y_0$ is 
the normalization. 

In Section \ref{st_step3}, we prove the following claim. 

\begin{claim}\label{fund_claim}
Let $i\geq 0$ such that $K_{Y_i}+\Delta_i$ is not nef over $X$. 
Assume that there exists a projective and birational morphism 
$\phi_i\colon Z_{i,0}\to\bar{Y}_i$
$($we note that $\nu_{Y_i}\colon\bar{Y}_i\to Y_i$ is the normalization$)$
and a $\Q$-divisor $H_{i,0}$ on $Z_{i,0}$ 
such that the following properties hold:
\begin{enumerate}
\renewcommand{\theenumi}{\roman{enumi}}
\renewcommand{\labelenumi}{(\theenumi)}
\item\label{fund_claim1}
The variety $Z_{i,0}$ is normal and $\Q$-factorial.
\item\label{fund_claim2}
$K_{Z_{i,0}}+B_{i,0}=\phi^*(K_{\bar{Y}_i}+B_i)$
holds, where $B_i$ is the sum of $D_{\bar{Y}_i}$ and the strict transform of $\Delta_i$, 
and $B_{i,0}$ is the strict transform of $B_i$. In particular, the pair 
$(Z_{i,0}, B_{i,0})$ has canonical singularities. 
\item\label{fund_claim3}
$H_{i,0}\sim_{\Q}(\nu_{Y_i}\circ\phi_i)^*H_i$ holds.
\end{enumerate}
Then we have the following results:
\begin{enumerate}
\renewcommand{\theenumi}{\arabic{enumi}}
\renewcommand{\labelenumi}{(\theenumi)}
\item\label{fund_claim4}
We can run $(K_{Z_{i,0}}+B_{i,0})$-MMP
\[
Z_{i,0}\xrightarrow{\pi_{i,0}}V_{i,0}\xleftarrow{\pi_{i,0}^+}Z_{i,1}
\xrightarrow{\pi_{i,1}}V_{i,1}\xleftarrow{\pi_{i,1}^+}Z_{i,2}
\xrightarrow{\pi_{i,2}}\cdots.
\]
over $\bar{W}_i$, where $\bar{W}_i$ is the normalization of $W_i$. 
Let $B_{i,j}$, $H_{i,j}$ be the push forward of $B_{i,0}$, $H_{i,0}$ on $Z_{i,j}$, 
respectively. 
More precisely, for $j\geq 0$, the morphism $\pi_{i,j}$ is the contraction morphism 
associated to a $(K_{Z_{i,j}}+B_{i,j})$-negative extremal ray 
$R_{i,j}\subset\overline{\NE}(Z_{i,j}/\bar{W}_i)$ and the morphism $\pi_{i,j}^+$ is the 
identity morphism if $\pi_{i,j}$ is divisorial and the flip if $\pi_{i,j}$ is small. 
\item\label{fund_claim5}
Let $\sL_{i,j}$ be the pullback of $\sL_i$ on $Z_{i,j}$. Then we have the following:
\[
\sL_{i,j}\sim_{\Q}
l_i(K_{Z_{i,j}}+B_{i,j}+\lambda_iH_{i,j}).
\]
\item\label{fund_claim6}
If $K_{Z_{i,j}}+B_{i,j}$ is not nef over $\bar{W}_i$, then 
we have the following: 
\[
\lambda_i=\inf\{\lambda\in\R_{\geq 0}\, |\, K_{Z_{i,j}}+B_{i,j}
+\lambda H_{i,j} \text{ is nef over }X\}.
\]
\item\label{fund_claim7}
Assume that the sequence 
\[
Z_{i,0}\dashrightarrow Z_{i,1}\dashrightarrow\dots\dashrightarrow Z_{i,m_i}
\]
terminates. In other words, $K_{Z_{i,m_i}}+B_{i,m_i}$ is nef over $\bar{W}_i$. 
$($We will prove the termination in Section \ref{st_step4}.$)$
Then $m_i\in\Z_{>0}$. 
Set $Z_{i+1,0}:=Z_{i,m_i}$ and $H_{i+1,0}:=H_{i,m_i}$. 
Then there exists a projective and birational morphism 
$\phi_{i+1}\colon Z_{i+1,0}\to\bar{Y}_{i+1}$ such that the properties 
\eqref{fund_claim1}, \eqref{fund_claim2} and \eqref{fund_claim3} holds. 
\end{enumerate}
\end{claim}

\begin{proof}[{Proof of Claim \ref{fund_claim}}]

Since the pair $(Z_{i,0}, B_{i,0})$ is $\Q$-factorial and has canonical singularities, 
we can run the MMP which described in \eqref{fund_claim4}. 
We note that the flip $\pi_{i,j}^+$ exists 
if $\pi_{i,j}$ is small by \cite[Corollary 1.4.1]{BCHM}. 

Now we prove \eqref{fund_claim5} and \eqref{fund_claim6} by induction on $j$. 
Let $\lambda_{i,j}$ be the right-hand of the equality in \eqref{fund_claim6}.
We consider the case $j=0$. 
By \eqref{smmp7} in Section \ref{st_step2}, $\sL_{i,0}$ is isomorphic to 
\[
\sO_{Z_{i,0}}(l_i(K_{Z_{i,0}}+B_{i,0}+\lambda_i(\nu_{Y_i}\circ\phi_i)^*H_i)).
\]
Thus we prove \eqref{fund_claim5} 
for the case $j=0$ since $H_{Z_{i,0}}\sim_{\Q}(\nu_{Y_i}\circ\phi_i)^*H_i$. 
On the other hand, 
\begin{eqnarray*}
\lambda_{i,0} & = & \inf\{\lambda\in\R_{\geq 0}\, |\, 
\phi_i^*(K_{\bar{Y}_i}+B_i+\lambda\nu_{Y_i}^*H_i)
\text{ is nef over }X\}\\
 & = & \inf\{\lambda\in\R_{\geq 0}\, |\, 
K_{Y_i}+\Delta_i+\lambda H_i
\text{ is nef over }X\}=\lambda_i.
\end{eqnarray*}
Thus we prove \eqref{fund_claim6} for the case $j=0$.

We consider the case $j\geq 1$. 
Since the inverse of the birational map $Z_{i,j-1}\dashrightarrow Z_{i,j}$
does not contract divisors, 
we prove \eqref{fund_claim5} by induction. 
We note that $\sL_{i,j}$ is nef over $X$ since $\sL_i$ is nef over $X$. 
Thus $\lambda_{i,j}\leq\lambda_i$ holds. 
On the other hand, we know that 
$K_{Z_{i,j}}+B_{i,j}$ is not nef over $\bar{W}_{i,j}$, 
$K_{Z_{i,j}}+B_{i,j}+\lambda_iH_{i,j}\sim_{\Q,\bar{W}_i}0$ and $K_{Z_{i,j}}+B_{i,j}+
\lambda_{i,j}H_{i,j}$ is nef over $\bar{W}_{i,j}$. 
Thus $\lambda_{i,j}\geq\lambda_i$ holds. 
Therefore we prove \eqref{fund_claim6}. 

Assume that the MMP 
\[
Z_{i,0}\dashrightarrow Z_{i,1}\dashrightarrow Z_{i,2}\dashrightarrow\cdots
\]
over $\bar{W}_i$ terminates and induces a minimal model 
$(Z_{i,m_i}, B_{i,m_i})$ over $\bar{W}_i$. 
We remark that 
$K_{Z_{i,0}}+B_{i,0}$ is not nef over $\bar{W}_i$ 
since $-(K_{Y_i}+\Delta_i)$ is ample over $W_i$ and 
$K_{Z_{i,0}}+B_{i,0}$ is the pullback of $K_{Y_i}+\Delta_i$. Thus $m_i\in\Z_{>0}$ holds.

\begin{claim}\label{s_claim}
$\bar{Y}_{i+1}$ is the 
canonical model of the pair $(Z_{i,0}, B_{i,0})$ over $\bar{W}_i$. 
\end{claim}

\begin{proof}[Proof of Claim \ref{s_claim}]
Let $g\colon T\to Z_{i,0}$ be a projective log resolution of the pair 
$(Z_{i,0}, B_{i,0})$ such that $\Supp B_T\subset T$ 
is a smooth divisor, where $B_T$ is the strict transform 
of $B_{i,0}$. 
Then, by Lemmas \ref{OX22} and \ref{OX33} \eqref{OX332}, $\bar{Y}_{i+1}$
is the canonical model of the pair $(T, B_T)$ over $\bar{W}_i$. 
We can write 
$K_T+B_T=g^*(K_{Z_{i,0}}+B_{i,0})+F$ with $F$ 
effective and exceptional over $Z_{i,0}$ since the pair $(Z_{i,0}, B_{i,0})$ has 
canonical singularities. 
Thus $\bar{Y}_{i+1}$ is the 
canonical model of $(Z_{i,0}, B_{i,0})$ over $\bar{W}_i$ by 
\cite[Corollary 3.53]{KoMo}.
\end{proof}

By Claim \ref{s_claim} and \cite[Theorem 2.22]{utah}, there exists the unique 
projective and birational morphism 
$\phi_{i+1}\colon Z_{i,m_i}\to\bar{Y}_{i+1}$ with 
$\phi_{i+1}^*(K_{\bar{Y}_{i+1}}+B_{i+1})
=K_{Z_{i,m_i}}+B_{i,m_i}$. 
Set $Z_{i+1,0}:=Z_{i,m_i}$ and $H_{i+1}:=H_{i,m_i}$. We note that $B_{i,m_i}=B_{i+1,0}$.
By Claim \ref{s_claim} and Section \ref{st_step2}, 
\begin{eqnarray*}
 &   & \sO_{Z_{i+1,0}}(l_i(K_{Z_{i+1,0}}+B_{i+1,0}+\lambda_iH_{i+1,0}))\\
 & \simeq & (\sL_i)_{Z_{i+1,0}}
\simeq((\nu_{W_i}\circ\bar{\pi}_i^+)^*\sL_i)_{Z_{i+1,0}}\\
 & \simeq & \sO_{Z_{i+1,0}}(l_i(K_{Z_{i+1,0}}+B_{i+1,0}+
\lambda_i(\nu_{Y_{i+1}}\circ\phi_{i+1})^*H_{i+1}))
\end{eqnarray*}
holds, where $(\sL_i)_{Z_{i+1,0}}$, $((\nu_{W_i}\circ\bar{\pi}_i^+)^*\sL_i)_{Z_{i+1,0}}$ 
is the pullback of $\sL_i$, $(\nu_{W_i}\circ\bar{\pi}_i^+)^*\sL_i$ to $Z_{i+1,0}$, 
respectively. 
Hence $H_{i+1,0}\sim_{\Q}(\nu_{Y_{i+1}}\circ\phi_{i+1})^*H_{i+1}$ holds.
Thus we prove \eqref{fund_claim7}. 
\end{proof}

Therefore, by Claim \ref{fund_claim}, if we assume the termination of the sequence 
in \eqref{fund_claim7}, then we can inductively construct the diagram 

\[\xymatrix{
Z_{i,0} \ar@{-->}[r] \ar[d]_{\phi_i} & Z_{i,1} \ar@{-->}[r] & {\cdots}\ar@{-->}[r] & Z_{i+1,0} \ar[d]^{\phi_{i+1}}\\
\bar{Y}_i \ar[d]_{\nu_{Y_i}} &        &              & \bar{Y}_{i+1} \ar[d]^{\nu_{Y_{i+1}}} \\
Y_i \ar[dr]_{\pi_i}         &            &              & Y_{i+1} \ar[dl]^{\pi_i^+}  \\
                      &  W_i               & W_i^d  \ar[l]^{d_i}     &                  \\
}\]
for any $i$ with $K_{Y_i}+\Delta_i$ not nef over $X$.

\subsection{Termination of the program}\label{st_step4}

In this section, we show the following: 

\begin{enumerate}
\renewcommand{\theenumi}{\alph{enumi}}
\renewcommand{\labelenumi}{(\theenumi)}
\item\label{step41}
For any $i\geq 0$ with $K_{Y_i}+\Delta_i$ not nef over $X$, 
the sequence $Z_{i,0}\dashrightarrow Z_{i,1}\dashrightarrow\cdots$ 
terminates. 
\item\label{step42}
The sequence $Y_0\dashrightarrow Y_1\dashrightarrow\cdots$ 
terminates. 
\end{enumerate}

Assume either \eqref{step41} does not hold for some $i$, or 
\eqref{step41} is true and \eqref{step42} does not hold. 
Then there exists an infinite sequence 
\[
Z_{0,0}\dashrightarrow Z_{0,1}\dashrightarrow\cdots\dashrightarrow 
Z_{1,0}\dashrightarrow Z_{1,1}\dashrightarrow\cdots
\]
of varieties $Z_{i,j}$ (we note that $m_i\in\Z_{>0}$ by Claim \ref{fund_claim} 
\eqref{fund_claim7}). 
For any $i$ and $j$, under the natural embedding 
\[
\NC(Z_{i,j}/\bar{W}_i)\hookrightarrow\NC(Z_{i,j}/\bar{X}),
\] 
the cone $\overline{\NE}(Z_{i,j}/\bar{W}_i)$ is an extremal face in 
$\overline{\NE}(Z_{i,j}/\bar{X})$. 
Hence the extremal ray $R_{i,j}\subset\overline{\NE}(Z_{i,j}/\bar{W}_i)$ can be seen as a 
$(K_{Z_{i,j}}+B_{i,j})$-negative extremal ray $R'_{i,j}\subset\overline{\NE}(Z_{i,j}/\bar{X})$. 
By Claim \ref{fund_claim} \eqref{fund_claim5}, 
$(K_{Z_{i,j}}+B_{i,j}+\lambda_iH_{i,j}\cdot R'_{i,j})=0$ holds. Moreover, 
\[
\lambda_i=\inf\{\lambda\in\R_{\geq 0}\, |\, K_{Z_{i,j}}+B_{i,j}
+\lambda H_{i,j} \text{ is nef over }\bar{X}\}
\]
holds by Claim \ref{fund_claim} \eqref{fund_claim6}. 
The contraction morphism with respect to $R'_{i,j}$ over $\bar{X}$ is equal to $\pi_{i,j}$. 
Thus this sequence is a 
$(K_{Z_{0,0}}+B_{0,0})$-MMP with scaling $H_{0,0}$ over $\bar{X}$. 
By Theorem \ref{BCHM_thm}, Lemma \ref{base_lem} and the facts
that the pair $(Z_{0,0}, B_{0,0})$ has canonical singularities and 
the morphism $Z_{0,0}\to\bar{X}$ is a projective log resolution of the pair 
$(\bar{X}, \Delta_{\bar{X}}+D_{\bar{X}})$, this MMP over $\bar{X}$ must terminates. 
This leads to a contradiction. 
Thus both \eqref{step41} and \eqref{step42} are true.

\subsection{Conclusion}\label{st_step5}

By Sections \ref{st_step2} and \ref{st_step4}, 
there exists a projective and birational morphism 
$f_m\colon Y_m\to X$ such that $K_{Y_m}+\Delta_m$ is nef over $X$. 
Furthermore, the pair $(Y_m, \Delta_m)$ is semi-terminal, the morphism $f_m$ is 
an isomorphism in codimension $1$ over $X$ and 
is an isomorphism at any generic point of $D_{Y_m}$ by construction. Therefore 
the morphism $f_m$ is a semi-terminal modification of $(X, \Delta)$. 

As a consequence, 
we have completed the proof of Theorem \ref{mainthm} \eqref{mainthm2}.

\end{document}